\documentclass[reqno]{amsart}

\usepackage{amsfonts}
\usepackage{amssymb}
\usepackage{color}
\usepackage{amsthm}
\usepackage{amsmath}
\usepackage{enumerate}
\usepackage[usenames,dvipsnames]{xcolor}
\usepackage{geometry}
\usepackage{enumitem}

\newtheorem{thm}{Theorem}
\newtheorem{lm}{Lemma}
\newtheorem{prop}{Proposition}

\theoremstyle{definition}
\newtheorem{cor}{Corollary}
\newtheorem{df}{Definition}

\theoremstyle{remark}
\newtheorem{rem}{Remark}

\newcommand{\R}{\mathbb{R}}
\newcommand{\PP}{\mathbb P}
\newcommand{\T}{\mathbb T}
\newcommand{\N}{\mathbb N}
\newcommand{\Z}{\mathbb Z}
\newcommand{\C}{\mathbb C}
\newcommand{\A}{\mathbb A}

\title{The Arnold conjecture in $\C\mathbb P^n$ and the Conley index}

\author[L. Asselle]{Luca Asselle}
\address{Ruhr-Universit\"at Bochum, Fakult\"at f\"ur Mathematik, Universit\"atsstr. 150, 44801, Bochum, Germany}
\email{luca.asselle@rub.de}

\author[M. Izydorek]{Marek Izydorek}
\address{Gda\'nks University of Technology, Gabriela Narutowicza 11/12, 80233 Gda\'nsk, Poland}
\email{marekizydorek@pg.edu.pl}

\author[M. Starostka]{Maciej Starostka}
\address{Gda\'nks University of Technology, Gabriela Narutowicza 11/12, 80233 Gda\'nsk, Poland \newline
Institut f\"ur Mathematik, Naturwissenschaftliche Fakult\"at II, Martin-Luther-Universit\"at Halle-Wittenberg, 06099 Halle (Saale), Germany}
\email{maciejstarostka@pg.edu.pl}
 
\begin{document}

\maketitle
%\tableofcontents
%
%
%{\color{red}
%\quad \\
%To add:
%\section{Introduction}
%\section{LS-Conley index}
%See \cite{gip}
%\section{Relative cup-length}
%See \cite{KGUss}
%\section{Hamiltonian systems on $T^{2n}$} 
%See \cite{StWa}
%  }

\begin{abstract}
In this paper we give an alternative, purely Conley index based proof of the Arnold conjecture in $\C\mathbb P^n$ asserting that a Hamiltonian diffeomorphism of $\C\mathbb P^n$ endowed with 
the Fubini-Study metric has at least $n+1$ fixed points.
\end{abstract}

%%%%%%%%%%%%%%
%%%%%%%%%%%%%%
%%%%%%%%%%%%%%

\section{Introduction}

The Arnold conjecture for closed symplectic manifolds, formulated by Arnold in 1960s, states that a Hamiltonian diffeomorphism $\phi$ of a closed connected symplectic manifold $(M,\omega)$ should have at least 
as many ``contractible'' fixed points as a smooth function $f:M\to \R$ has critical points. Such a conjecture can be seen as a natural generalization of Poincar\'e's last geometric theorem and represents one of the most famous (and still nowadays open in its full generality) problems in symplectic geometry. 

Thus, let $(M,\omega)$ be a closed connected symplectic manifold, and let $H:\T\times M\to \R$, $\T:=\R/\Z$, be a smooth time-dependent one-periodic Hamiltonian function on $M$. 
The Hamiltonian $H$ induces a family $\{\phi_t\}_{t\in \R}$ of symplectic (meaning that $\phi_t^*\omega=\omega$ for all $t\in \R$) diffeomorphisms of $M$, which is usually referred to as the 
flow (actually, $\phi_t$ is not a flow 
as the identity $\phi_s\circ \phi_t = \phi_{s+t}$ is usually not satisfied) of $H$, by 
$$\left \{ \begin{array}{l} \displaystyle \frac{\mathrm d}{\mathrm d t} \phi_t = X_H (\phi_t),\\ \ \ \ \, \phi_0 = \text{id},\end{array}\right .$$ 
where $X_H$ is the (time-dependent) Hamiltonian vector field defined by 
$$\iota_{X_H} \omega = \omega(X_H,\cdot) = - \mathrm d H.$$
A diffeomorphism $\phi$ of $M$ is called \textit{Hamiltonian} if there exists a time-dependent Hamiltonian $H:\T\times M\to \R$ such that $\phi=\phi_1$, i.e. such that the time-one-Hamiltonian flow of $H$ is given by $\phi$.
A fixed point $x\in M$ of a Hamiltonian diffeomorphism $\phi$ is called \textit{contractible} if $[0,1]\ni t\mapsto \phi_t(x)$ is a contractible loop in $M$.

There are many (non-equivalent) ways to state the Arnold conjecture, for instance as an inequality involving the number $|\text{Fix} (\phi)|$ of contractible fixed points of $\phi$ and the minimal number $\text{crit} (M)$ 
of critical points of a smooth function on $M$, the Lyusternik-Schnirelmann category $LS(M)$ of $M$, or the cup-length $CL(M)$ of $M$ (notice that $\text{crit}(M)\geq LS(M)\geq CL(M)$). 
In this setting the most general answer appears in the work of Rudyak-Oprea \cite{Rudyak:99}, which shows that 
$$|\text{Fix}(\phi)| \geq \text{crit}(M)$$
for any closed symplectically aspherical manifold $(M,\omega)$. 

Under the additional assumption that the Hamiltonian diffeomorphism $\phi$ be non-degenerate (meaning that 1 is not an eigenvalue of the linearized map $d\phi$ at 
each fixed point) one can state the Arnold conjecture for instance as 
$$|\text{Fix}(\phi)| \geq \text{Morse}(M) := \min \{  |\text{crit}(f)| \ |\ f \ \text{Morse function on} \ M\}.$$
Arnold noted that such a statement, usually known as the strong Arnold conjecture, holds true if the Hamiltonian is $C^2$-small. For closed symplectic surfaces it was 
proved by Eliashberg \cite{Eliashberg}, for the $2m$-Torus $\mathbb T^{2m}$ equipped with the standard symplectic structure by Conley and Zehnder \cite{Conley} (see also \cite[Chapter 6]{Hofer}), 
and for the complex projective space $\mathbb C\mathbb P^n$ 
equipped with the Fubini-Study form by Fortune \cite{Fo}. Despite these contributions, the strong Arnold conjecture is still open. 
The weaker statement known as homological Arnold conjecture, which reads
$$|\text{Fix}(\phi)| \geq \sum_{i=0}^{\dim M} \dim H_i(M;\mathbb Q),$$
has been proved by Floer \cite{Floer:88} in the case of symplectically aspherical manifolds. 
The validity of the homological Arnold conjecture has been then extended to general closed symplectic manifolds by Floer \cite{Floer:89}, Fukaya-Ono \cite{Fukaya:1,Fukaya:2}, 
Liu-Tian \cite{Liu}, and Ruan \cite{Ruan}, building on the original work of Floer. For a comprehensive discussion about all variants of the Arnold conjecture as well as about the state of the art on each of them we refer to the expository note \cite{Golovko}. %The techniques introduced by Floer in its proof of the Arnold conjecture 
%have turned out to be extremely powerful and have been since then applied in a huge variety of situations, thus going way beyond Floer's original motivation.  

In the degenerate case, the Arnold conjecture is still open (in any of the form stated above) for general non symplectically aspherical manifolds, 
e.g. for toric manifolds, and even in rather simple cases such as the product 
$\T^{2m}\times \C\mathbb P^n$ equipped with the standard product symplectic form. In the latter case, the best known result is due to Oh \cite{Oh} and states that 
any Hamiltonian diffeomorphism on $\T^{2m}\times \C\mathbb P^n$ has at least $n+1$ contractible fixed points. There, the author also raises doubts on the validity of the degenerate Arnold conjecture (in any of its form) 
in full generality. For other partial results we refer e.g. to \cite{Givental,Schwarz}. In any case, the degenerate Arnold conjecture seem at the moment out of the reach of Floer theoretical methods. With this in mind, the 
authors in collaboration with Abbondandolo started few years ago a program aimed at developing alternative and complementary approaches 
to Hamiltonian Floer theory, both for abstract functionals and for the Hamiltonian action functional on the loop space of certain symplectic manifolds: In the paper \cite{Asselle} and the subsequent preprint \cite{Abbo} 
the first and third named authors together with Abbondandolo define a Morse complex for the Hamiltonian action functional over a ``mixed regularity'' space of loops in the cotangent bundle $T^*M$ of a closed manifold $M$. 
As one expects, the resulting homology is isomorphic to the Floer homology of $T^*M$ as well as to the singular homology of the free loop space of $M$. The advantage of such an approach over Floer homology relies 
on the fact that having a globally defined gradient flow with nice properties allows a much larger class of perturbations which in turns reflects into the fact that transversality issues can be more easily overcome than in Floer theory. Also, the gradient flow equation need not necessarily be local, and this allows to treat also non-local problems. 

In the present paper we give a purely Conley index based proof of the Arnold conjecture in $\C\mathbb P^n$.
This goes the direction of developing methods alternative to Floer theory which are more topological in nature, and which can be applied e.g. to the Hamiltonian action on the space of loops on toric manifolds, with the concrete motivation of improving the known results about the degenerate Arnold conjecture on such a class of symplectic manifolds. 
The advantage relies here on the fact that the resulting theory should be substantially richer than Floer homology. Indeed, there are certain constructions which seem to be too complicated to be obtained by Floer homology, due among
other problems to the already mentioned transversality issues, but which can be obtained by a theory which allows the use of arguments from Conley theory and homotopy. Examples can be found for instance in the work of Kragh \cite{Kragh1,Kragh2} and have led to a proof of a homotopy version of the nearby Lagrangian conjecture. 

\begin{thm}[\cite{Fo}]
Let $\phi$ be a Hamiltonian diffeomorphism of $\C\mathbb P^n$ endowed with the Fubini-Study symplectic form. Then 
$$|\mathrm{Fix}(\phi)| \geq n+1.$$ 
\label{thm:main}
\end{thm}

\vspace{-4mm}

We finish this introduction with a brief summary of the content of the paper: 
\begin{itemize}
\item In Section 2 we recall the definition of the Conley index for a suitable class of gradient flows on a Hilbert space as well as the notion of relative cup-length. 
Also, we recall the correspondence between the fixed points of a Hamiltonian diffeomorphism of $\C\mathbb P^n$ and the critical points of a suitably defined Hamiltonian action over the space of loops in 
$\R^{2n+2}$ with Sobolev regularity $1/2$. 
\item In Section 3, we introduce the abstract notion of IA-homotopy and show that the Conley index 
of the maximal bounded invariant set for a suitable class of flows in the product space $H^{1/2}(\T,\R^{2n+2})\times \R$ 
depends only on the IA-homotopy classes of the end-points. 
\item In Section 4 we first sketch the proof of Theorem~\ref{thm:main} in the easier case of $C^0$-small Hamiltonians, and then prove Theorem \ref{thm:main} for an arbitrary Hamiltonian. We shall stress already at this point that proving the Arnold conjecture for $C^0$-small Hamiltonians is not enough since the Hofer diameter of $\C\mathbb P^n$ is infinite, see \cite[Section 6]{Polterovich}. On the other hand, the Arnold conjecture in the $C^0$-small case does not follow from Arnold's original observation, which applies only if the Hamiltonian is $C^2$-small. 
\end{itemize}

\vspace{2mm}

\noindent \textbf{Acknowledgments:} We thank Alberto Abbondandolo for many useful discussions. This research is supported by the DFG-project 380257369 ``Morse theoretical methods
in Hamiltonian dynamics''. M.I. is supported by the Beethoven2-grant 2016/23/G/ST1/04081 of
the National Science Centre, Poland. M.S. is supported by the DFG-grant 459826435 ``The equivariant spectral flow and bifurcation for indefinite functionals with symmetries''.

%%%%%%%%%%%%%%
%%%%%%%%%%%%%%
%%%%%%%%%%%%%%

\section{Preliminaries}

%%%%%%%%%%%%%%

\subsection{Conley index in a Hilbert setting} In this subsection we recall how to define a notion of Conley index for gradient flows of the form ``compact perturbation of a fixed Fredholm operator'' on a Hilbert space, which for convenience will be always assumed to be the space of Sobolev loops of regularity $H^{1/2}$ in $\R^{2n+2}$. The main reference is here \cite{gip}. For the classical definition of the Conley index on a locally compact 
metric space we refer e.g. to \cite{Conley:78,Salamon}. Thus, we %denote by $\{e_i\}$ the standard Hilbert basis of $H^{1/2}(\T,\R^2)$ and 
set 
 $$E:= H^{1/2}(\T,\R^{2n+2}).$$
 and endow it with the standard $H^{1/2}$-scalar product, 
 and consider the canonical orthogonal splitting 
 $$E \cong E^+\oplus E^-\oplus E^0,$$
 where 
 $$E^+:= \Big \{x= \sum_{k\in \N} e^{2\pi k J t}\, x_k \in E \ \Big | \ x_k\in \R^{2n+2}\Big \}, \quad E^-:= \Big \{ x= \sum_{k\in \N} e^{-2\pi k J t} x_k \in E \ \Big |\ x_k \in \R^{2n+2}\Big \},$$
 and $E^0\cong \R^{2n+2}$ is the space of constant loops. Here, $J$ denotes the standard complex structure of $\R^{2n+2}$.
 We define 
\begin{equation}
\label{L}
L:E\to E,\quad L x =  L(x^+ + x^-+x^0) := x^+-x^-,
\end{equation}
 where $x^\pm,x^0$ are the orthogonal projections of $x\in E$ onto $E^\pm,E^0$ respectively, and consider
 gradient vector fields of the form
 \begin{equation}
 \label{functional}
 F= L+K
 \end{equation}
 where $K$ is Lipschitz continuous and maps bounded sets into relatively compact sets. In the literature vector fields of this kind are usually referred to as $\mathcal L\mathcal S$-\text{vector fields}. 
 We further denote the flow generated by the vector field $F$ with $\eta$ and call it hereafter an $\mathcal L \mathcal S$-\text{flow}. 
 
 \begin{df} Let $\eta$ be an $\mathcal L \mathcal S$- flow on $E$. A closed and bounded set $\Omega \subset E$ is called an \textit{isolating neighborhood} for $\eta$ if 
 $$\text{inv}\, (\Omega,\eta) := \Big \{x\in \Omega\ \Big |\ \eta(t,x )\in \Omega, \ \forall t \in\R\Big \} \subset \text{Int}\, \Omega.$$
 An invariant set $S$ for the flow $\eta$ is called:
 \begin{itemize}
 \item an \textit{isolated invariant set} if there exists an isolating neighborhood $\Omega$ such that $\text{inv} \, (\Omega,\eta) = S$,
 \item a \textit{maximal bounded invariant set} if there exists $R_0>0$ such that $\text{inv}\, (B(R),\eta)=S$ for every $R\geq R_0$, where $B(R)$ denotes the ball with radius $R$ around the origin in $E$. 
 \end{itemize}
Clearly, a maximal bounded invariant set is isolated. 
\end{df}

For $\mathcal L \mathcal S$-flows there is a well defined notion of Conley index, as we now recall.  
To this purpose we start observing that an isolating neighborhood $\Omega$ yields an isolating neighborhood for every sufficiently large finite dimensional approximation of $E$. In what follows we consider the following finite 
dimensional subspaces of $E$: for $k,l\in \N$ we set 
$$E_k := \Big \{ e^{2\pi k J t} q \ \Big |\ q \in \R^{2n+2}\Big \}, \quad E^{-k,l} := \bigoplus_{i=-k}^l E_i,$$
and denote by 
$$A^{-k,l} := A \cap E^{-k,l}, \quad \forall A\subset E.$$
For any $\mathcal L\mathcal S$-vector field $F$  we further set 
$$F^{-k,l}:= \big ( L + \pi^{-k,l} \circ K\big )  \Big |_{E^{-k,l}}, $$
where $\pi^{-k,l}:E\to E$ is the orthogonal projection onto $E^{-k,l}$, and denote by $\eta^{-k,l}$ the induced flow.

\begin{lm}[\cite{gip}, Lemma 4.1]\label{lem:1}
Let $\eta$ be an $\mathcal L\mathcal S$-flow, and suppose that $\Omega$ is an isolating neighborhood for $\eta$. Then there exists 
$n_0\in \N$ such that for every $k,l>n_0$ the set $\Omega^{-k,l}$ is an isolating neighborhood for $\eta^{-k,l}.$ \qed
\end{lm}

Roughly speaking, the positive integer $n_0$ is chosen in such a way that the scalar product of the vector fields $L+K$ and $L$ is positive on the orthogonal complement of $E^{-n_0,n_0}$.
The lemma above implies that for any isolating neighborhood $\Omega$ for an $\mathcal L \mathcal S$-flow $\eta$ we have a family of Conley indices 
$$h^{-k,l}(\Omega,\eta) := h^{-k,l}(\Omega^{-k,l},\eta^{-k,l}),\qquad k,l>n_0.$$
As it is shown in \cite{gip}, such homotopy types stabilize in the following sense. 

\begin{lm}\label{lem:2}
Let $\eta,\Omega$ be as in the statement of Lemma \ref{lem:1}. Then there exists $n_0\in \N$ such that for any $k'>k>n_0$ and $l'>l>n_0$ we have 
$$h^{-k',l'}(\Omega,\eta) = S^{(2n+2)(l'-l)}\cdot h^{-k,l}(\Omega,\eta),$$
meaning that the Conley index $h^{-k',l'}(\Omega,\eta)$ is obtained from $h^{-k,l}(\Omega,\eta)$ by a suspension of dimension equal to the dimension of $E^{l+1,l'}$. \qed
\end{lm}

\begin{rem}
By the continuation property of the Conley index, Lemmas \ref{lem:1} and \ref{lem:2} generalize to families $\{\eta_s\}_{s\in [0,1]}$ of $\mathcal L \mathcal S$-flows admitting a common isolating neighborhood $\Omega$. \qed
\end{rem}

%%%%%%%%%%%%%

\subsection{Relative cup-length of the Conley index} 

In this subsection we recall the notion of relative cup-length for the Conley index in locally compact metric spaces following \cite{KGUss}. Such a notion already appears in \cite{Floer:87} (see also \cite[Page 598]{Floer:89} for an analogous 
construction in the loop space) and is used to prove a global and topological continuation theorem for normally hyperbolic invariant sets. 
Thus, let $A\subset X\subset Y$ be compact metric spaces. 
Denoting the Alexander-Spanier cohomology by $H^*$, we easily see that $H^*(X,A)$ has a natural structure of $H^*(Y)$-module with multiplication induced by the cup-product via
$$\beta \cdot \alpha := \iota^* \beta \cup \alpha, \qquad \forall \alpha \in H^*(X,A), \, \forall \beta \in H^*(Y),$$
where $\iota:X\to Y$ denotes the standard inclusion. 

\begin{df}
The \textit{relative cup-length} $\mathcal Y(X,A;Y)$ of the $H^*(Y)$-module $H^*(X,A)$ is equal to:
\begin{itemize}
\item $0$, if $H^*(X,A)=0$.
\item $1$, if $H^*(X,A)\neq 0$ and 
$$\beta\cdot \alpha = 0,\qquad \forall \alpha \in H^*(X,A), \, \forall \beta \in H^{>0}(Y).$$
\item $k\geq 2$, if there exists $\alpha_0\in H^*(X,A)$ and $\beta_1,...,\beta_{k-1}\in H^{>0}(Y)$ such that 
$$(\beta_1 \cup ... \cup \beta_{k-1})\cdot \alpha_0 \neq 0,$$
and 
$$(\gamma_1\cup ... \cup \gamma_k )\cdot \alpha =0, \qquad \forall \alpha \in H^*(X,A), \ \forall \gamma_1,...,\gamma_k \in H^{>0}(Y).$$
\end{itemize}
\end{df}

\begin{rem}
In \cite{KGUss} only powers of one single $\beta\in H^*(Y)$ were considered in the definition of the relative cup-length, the reason being the fact that this was enough 
for the applications the authors had in mind (indeed, the cup-length of $\mathbb C \mathbb P^n$ is clearly realized taking powers of a single cohomology class). 
For future reference it is here more convenient to give a general definition of relative cup-length which allows to consider product of different cohomology classes in $H^*(Y)$. 
It is easy to check that all properties of the relative cup-length discussed in \cite{KGUss} remain unchanged with this more general definition.
\qed
\end{rem}

With any isolating neighborhood $\Omega$ for a flow $\phi$ on a locally compact metric space (for our purposes, $\phi$ will always be of the form $\eta^{-k,l}$ for some $\mathcal L\mathcal S$-flow $\eta$ on $E$) 
we can associate a relative cup-length as follows: if $(N,L)$ is an index pair for $S:=\text{inv}\, (\Omega,\phi)$, we set 
$$\mathcal Y(\Omega,\phi) := \mathcal Y(N,L;\Omega).$$
Lemma 3.3 in \cite{KGUss} shows that such a definition is well-posed, namely independent of the choice of the index pair $(N,L)$, and that the relative cup-length is invariant under continuation. 

The importance of $\mathcal Y(\Omega,\phi)$ relies on the next theorem which allows us to estimate from below the number of elements in a Morse decomposition of the invariant set $S$. In case of gradient-like 
vector fields (which is the case we are interested in) this also provides a lower bound on the number of rest points.
%, {\color{red}and hence - in case $\phi$ is a gradient flow - also the number of rest points $\phi$ (indeed, in this case any element in a Morse decomposition yields at least one rest point)}.  
Recall that a \textit{Morse decomposition} of $S$ is a finite collection $M_1,...,M_k$ of disjoint compact invariant subsets of $S$ which can be ordered in such a way that the following holds: 
if
$$x\in S \setminus \bigcup_{i=1}^k M_i,$$
then there exist $i<j$ such that $\omega(x)\subset M_i$ and $\alpha (x)\subset M_j$, where as usual $\omega(x)$ resp. $\alpha(x)$ denotes the $\omega$-limit set resp. the $\alpha$-limit set of $x$. 
The following theorem is stated in \cite{KGUss} in a slightly weaker form (see \cite[Theorem 4.1]{KGUss}). However, the proof actually shows the slightly stronger version we state here and for this reason will be omitted. 

\begin{thm}
\label{thm:1}
Let $X$ be a locally compact metric space, $\Omega \subset X$ be an isolating neighborhood for the gradient-like flow $\phi$, $S:=\mathrm{inv}\, (\Omega,\phi)$. If $M_1,...,M_k$ is a Morse decomposition of $S$ 
such that the $M_i$'s admit pairwise disjoint and contractible isolating neighborhoods $\Omega_i$ with $M_i=\mathrm{inv}\, (\Omega_i,\phi)$, then 
\begin{equation*}
\hspace{63mm} k\geq \mathcal Y(\Omega,\eta). \, \hspace{63mm} \qed
\end{equation*}
\end{thm}

%\begin{proof}
%The quite simple proof is essentially based on the following observation: If $\Omega'\subset \Omega$ are two isolating neighborhoods for $S$, then 
%$$\mathcal Y(\Omega',\phi) \geq \mathcal Y(\Omega,\phi).$$
%This follows from the definition of the relative cup-length combined with the naturality of the cup-product. %In particular, we have
%$$\mathcal Y(\Omega,\eta) \leq \mathcal Y\mathcal (\bigcup_{i=1}^k \Omega_i,\eta) = k.$$
%Now take a sequence $\Omega = \mathcal U_0 \supset \mathcal U_1 \supset \mathcal U_2 \supset ...$ of isolating neighborhoods for $S$ such that $\cap_j \mathcal U_j =S$. 
%By the observation above we have
%$$\text{CL}(\mathcal U_j) \geq k-1,\quad \forall j \in \N,$$
%and hence by the continuity of the Alexander-Spanier cohomology we can infer that $\text{CL}(S)\geq k-1$, which in turn implies that the Lyusternik-Schnirelmann category of $S$ is greater or equal to
%$k$. 
%\end{proof}

The following properties of the relative cup-length will be useful later on. Their easy proof is left to the reader. 

\begin{lm}
\label{lem:3}
Let $X$ be a locally compact metric space, $\Omega \subset X$ be an isolating neighborhood for the flow $\phi$, $S:=\mathrm{inv}\, (\Omega,\phi)$. Then the following statements hold:
\begin{enumerate}
\item If $\Omega'\subset \Omega$ is another isolating neighborhood for $S$, then
$$\mathcal Y(\Omega',\phi)\geq \mathcal Y(\Omega,\phi).$$
\item If $\Omega'\subset \Omega$ is another 
isolating neighborhood for $S$ and the inclusion $\iota:\Omega'\to \Omega$ induces a surjective homomorphism in cohomology, then 
$$\mathcal Y(\Omega',\phi)= \mathcal Y(\Omega,\phi).$$
\item If $\eta=(\eta_1,\eta_2)$ is a product flow, $\Omega=\Omega_1\times \Omega_2$, then 
\begin{equation*} \hspace{37mm} \mathcal Y(\Omega,\phi) = \mathcal Y(\Omega_1,\eta_1)  \cdot \mathcal Y(\Omega_2,\eta_2). \hspace{37mm} \qed \end{equation*}
\end{enumerate}
\end{lm}

Even if the Hilbert space $E$ is not locally compact, in case of an $\mathcal L \mathcal S$-flow we can still use the relative cup-length of the Conley index and Theorem \ref{thm:1} by replacing $\eta$ with 
a suitably large (in the sense of Lemmas \ref{lem:1} and \ref{lem:2}) finite dimensional approximation $\eta^{-k,l}$. Indeed, as the next lemma states, a Morse decomposition of an isolated invariant set $S$ 
for $\eta$ always yields a Morse decomposition for the finite dimensional approximation $\eta^{-k,l}$. For the proof we refer to \cite[Theorem 4.2]{IzydorekJDE}. 

\begin{lm}
\label{lem:4}
Let $\eta$ be an $\mathcal L\mathcal S$-flow on $E$ and let $\Omega$ be an isolating neighborhood with $\mathrm{inv}\, (\Omega,\eta)=:S$. Suppose that $M_1,...,M_r$ is a Morse decomposition of $S$
such that the $M_i$'s admit pairwise disjoint isolating neighborhoods $\Omega_i$ with $M_i=\mathrm{inv}(\Omega_i,\eta)$. Then there exists $n_0\in \N$ such that if $k,l>n_0$ then $\Omega^{-k,l},\Omega_i^{-k,l}$ 
are isolating neighborhoods for $\eta^{-k,l}$ and $\mathrm{inv}(\Omega_1^{-k,l},\eta^{-k,l}),..., \mathrm{inv}(\Omega_r^{-k,l},\eta^{-k,l})$ is a Morse decomposition for $\mathrm{inv}(\Omega^{-k,l},\eta^{-k,l})$. \qed
\end{lm}

%%%%%%%%%%%%%%

\subsection{Hamiltonian systems on $\C\PP^n$} Consider a time-depending one-periodic Hamiltonian function  $H_0:\T\times \C\PP^n\to \R$, where $\C\PP^n$ is as usual equipped with the Fubini-Study symplectic form. The Arnold conjecture on $\C\PP^n$ (see \cite{Fo}) states that the number of one-periodic solutions to 
\begin{equation}
\dot x(t) = J \nabla H_0(t,x(t))
\label{eq:Hamiltoncpn}
\end{equation}
is bounded from below by $n+1$, namely the cup-length (and, actually, also the sum of Betti numbers) of $\C\PP^n$. Hereafter, for sake of simplicity we will refer to one-periodic solutions to \eqref{eq:Hamiltoncpn} simply as \textit{solutions}. Solutions to \eqref{eq:Hamiltoncpn} are more easily studied by lifting the Hamiltonian system to $\R^{2n+2}$ as we now recall (for more details we refer to \cite{Ab,Fo}). We lift $H_0$ to a one-periodic function on $S^{2n+1}$, denoted $H_1$, and then extend $H_1$ to $\R^{2n+2}$ quadratically
$$H:\T\times \R^{2n+2} \to \R,\qquad  H(t,x) := |x|^2 \cdot H_1 \Big (\frac{x}{|x|}, t\Big ),$$
where here $|\cdot|$ denotes the euclidean norm. Notice that, being $H$ 2-homogeneous resp. $S^1$-invariant, its gradient is 1-homogeneous resp.  $S^1$-equivariant. 
By construction, solutions to 
\begin{equation}
\dot x(t) = J \nabla H(t,x(t))
\label{eq:Hamiltonrn}
\end{equation}
descend to solutions to \eqref{eq:Hamiltoncpn}: given a solution $x$ to \eqref{eq:Hamiltonrn}, one easily sees that $x/\|x\|_2\subset S^{2n+1}$ is again a solution, and composing with the projection $S^{2n+1}\to \C\PP^n$ yields the desired solution to \eqref{eq:Hamiltoncpn}. In contrast, solutions to \eqref{eq:Hamiltoncpn} need not lift to solutions of \eqref{eq:Hamiltonrn}; indeed, they in general lift to paths which solve \eqref{eq:Hamiltonrn} but a priori only start and end at the same $S^1$-fibre of the Hopf fibration. 
To detect such more general solutions to \eqref{eq:Hamiltonrn} we consider for arbitrary $\lambda \in \R$ the  solutions to 
\begin{equation}
\dot y(t) = J \big ( \nabla  H(t,y(t)) - 2 \pi \lambda y(t)\big ).
\label{eq:Hamiltonlambda}
\end{equation}
Indeed, a straightforward computation shows that $y$ is a solution to \eqref{eq:Hamiltonlambda} if and only if 
$$x(t) := e^{2\pi \lambda Jt} y(t)$$ 
solves \eqref{eq:Hamiltonrn} and satisfies $x(1)= e^{2\pi \lambda J}x(0)$. Notice finally that the correspondence between solutions to \eqref{eq:Hamiltoncpn} and solutions to \eqref{eq:Hamiltonlambda} is not one-to-one: each solution to \eqref{eq:Hamiltoncpn} lifts to a $\C\times \Z$-family of solutions to \eqref{eq:Hamiltonlambda} in 
the following way: for a given solution $\bar x$ to \eqref{eq:Hamiltoncpn} pick a lift $x$ to $\R^{2n+2}$. By construction, $x$ satisfies \eqref{eq:Hamiltonlambda} for some $\lambda \in \R$; then
for any $\theta \in \R, c\in \R, k\in \Z$, the curve
$$t \mapsto c\cdot e^{2\pi J(\theta + kt)}x(t)$$
satisfies \eqref{eq:Hamiltonlambda} for $\tilde \lambda = \lambda + k$ and projects to $\bar x$. Conversely, each solution of \eqref{eq:Hamiltonlambda} projecting to $\bar x$ must be of this form.  
Since the ``scaling-factor'' $c$ is easily ruled out by imposing that the $L^2$-norm of the solution is one, i.e. $\|x\|_2=1$,  in order to find $n+1$ distinct solutions to \eqref{eq:Hamiltoncpn} we 
need to show that there exist $n+1$ distinct $S^1$-families of  solutions to 
\begin{equation}
\label{eq:Hamiltonlambdasphere} 
\left \{\begin{array}{l} \dot x(t) = J \big ( \nabla H(t,x(t)) - 2 \pi \lambda x(t)\big ), \\ \|x\|_2 =1,\end{array}\right .
\end{equation}
such that the corresponding values of $\lambda$ are all contained in an interval of the form $[\lambda_0,\lambda_0+1)$. 
Actually, it is not restrictive to assume further that all $\lambda$ shall lie in the open interval 
$(\lambda_0,\lambda_0+1)$: indeed, if \eqref{eq:Hamiltonlambdasphere} had solutions for every $\lambda$, then we would immediately have the 
existence of infinitely many solutions to \eqref{eq:Hamiltoncpn}. Notice that \eqref{eq:Hamiltonlambdasphere} has no  solutions if and only if \eqref{eq:Hamiltonlambda} has no non-zero solutions.

The nice feature of  solutions to \eqref{eq:Hamiltonlambdasphere} is that they admit a variational characterization as critical points of a suitable modification of the Hamiltonian 
action functional as we now show.
We start by setting 
\begin{equation}
F_\lambda (x) := - J \dot x - \nabla  H(x) + 2\pi \lambda x, \qquad \forall \lambda\in \R,
\label{eq:Flambda}
\end{equation}
and choose $\lambda_0\in \R$ such that \eqref{eq:Hamiltonlambdasphere} has no solutions for $\lambda=\lambda_0$, or, equivalently, such that $F_{\lambda_0}(x)=0$ implies $x=0$.  In what follows we denote by 
$$\jmath^*:L^2(\T,\R^{2n+2})\to H^{1/2}(\T,\R^{2n+2})=:E$$
the adjoint operator to the canonical inclusion, and with $\|\cdot\|$ the $H^{1/2}$-norm. Recall that 
$$\jmath^* \Big (- J \frac{\mathrm d}{\mathrm d t} \Big ) x = x^+ - x^-=Lx,$$
where $L$ is the operator defined in \eqref{L} and $x=x^++x^-+x^0$ is the canonical decomposition of $x$.

\begin{lm}
\label{lem:5}
There exists $\epsilon>0$ such that 
$$\inf_{\|x\|=1} \|\jmath^* F_\lambda (x) \| >\epsilon, \quad \forall \lambda \in (\lambda_0-\epsilon,\lambda_0+\epsilon).$$
As a corollary, \eqref{eq:Hamiltonlambdasphere} has no solutions (equivalently, \eqref{eq:Hamiltonlambda} has no non-zero solutions) for every $\lambda \in (\lambda_0-\epsilon,\lambda_0+\epsilon)$. 
\end{lm}
\begin{proof}
By assumption $\jmath^*F_{\lambda_0}(x)\neq 0$ for every $x\neq 0$. Suppose by contradiction that 
$$\inf_{\|x\|=1} \|\jmath^* F_{\lambda_0} (x) \| =0$$
and choose a sequence $\{x_n\}$ such that $\|x_n\|=1$ for all $n\in\N$ and $\jmath^* F_{\lambda_0}(x_n)\to 0$. Since 
$$\jmath^* F_{\lambda_0}(x_n) = Lx_n + \underbrace{\jmath^*\big (- \nabla H(x_n) + 2\pi \lambda_0 x_n\big )}_{=:K_{\lambda_0}(x_n)}$$
where $K_{\lambda_0}$ is a compact operator, we deduce that 
$$Lx_n = x_n^+-x_n^-$$ 
converges up to a subsequence. This implies that $x_n^+,x_n^-$ converge up to a subsequence. As $x_n^0$ lives in a finite dimensional space, up to extracting a further subsequence we also
have that $x_n^0$ converges. Summing up, $x_n$ converges up to a subsequence to some $x_\infty$ such that $\|x_\infty\|=1$ and $\jmath^* F_{\lambda_0}(x_\infty)=0$, clearly a contradiction. 
The claim follows now by the continuity of $\lambda\mapsto \jmath^* F_\lambda$.
\end{proof}

In order to detect only solutions whose corresponding value of $\lambda$ lie in $(\lambda_0,\lambda_0+1)$ as critical point of the Hamiltonian action we proceed as follows: for $\epsilon>0$ as 
in Lemma~\ref{lem:5} we consider a smooth function $\chi=\chi_\epsilon : \R\to \R$ satisfying the following properties:
\begin{enumerate}
\item[($\chi.0$)] $\chi(\lambda_0)=\lambda_0$, $\chi(\lambda_0+1)=\lambda_0+1$. 
\item[($\chi.1$)] $\chi(\lambda) = \lambda$ on $(\lambda_0+\epsilon,\lambda_0+1-\epsilon)$. 
\item[($\chi.2$)] $\chi'>0$ on $(\lambda_0,\lambda_0+1)$
\item[($\chi.3$)] $\chi' \equiv 0, \ \text{on} (-\infty,\lambda_0]\cup [\lambda_0+1,+\infty)$.
\end{enumerate}

We now define 
$$\A_H:E\times \R \to \R,\qquad \A_H(x,\lambda) := \frac 12 \langle -J\dot x,x\rangle_{2} - \int_0^1 H(t,x(t))\, \mathrm d t + \ \pi \big (\chi(\lambda ) \|x\|_2^2 - \lambda \big ).$$
It is straightforward to check that the gradient of $\A_H$ with respect to the product metric on $E\times \R$ is given by 
$$\nabla \A_H(x,\lambda) = \Big ( \jmath^*\big(-J\dot x - \nabla H(x) + 2 \pi \chi(\lambda) x\big ), \pi\big (\chi'(\lambda) \|x\|_2^2 - 1\big )\Big ).$$
Moreover, the following holds.
\begin{lm}
$x$ solves \eqref{eq:Hamiltonlambdasphere} for some $\lambda \in (\lambda_0,\lambda_0+1)$ if and only if $(x,\lambda)$ is a critical point of $\A_H$. 
\label{lem:6}
\end{lm}
\begin{proof}
If $x$ solves \eqref{eq:Hamiltonlambdasphere} for some $\lambda\in (\lambda_0,\lambda_0 + 1)$, then by Lemma~\ref{lem:5} we have that $\lambda \in (\lambda_0+\epsilon,\lambda_0+1-\epsilon)$. Property ($\chi.1$) now implies that 
$$\nabla \A_H(x,\lambda ) = \Big ( \jmath^*\big(-J\dot x - \nabla H(x) + 2 \pi \lambda x\big ), \pi\big ( \|x\|_2^2 - 1\big )\Big ) = (0,0),$$
that is, $(x,\lambda)$ is a critical point of $\A_H$. 
Conversely, let $(x,\lambda)$ be a critical point of $\A_H$. By Property ($\chi.3$) we have that $\lambda \in (\lambda_0,\lambda_0+1)$. If $\lambda\in (\lambda_0,\lambda_0+\epsilon)$, then using 
($\chi.0$)-($\chi.2$) we obtain that $\chi(\lambda)= \bar \lambda \in (\lambda_0,\lambda_0+\epsilon)$ and hence the first component of $\nabla \A_H(x,\lambda)$ is given by $\jmath^* F_{\bar \lambda}(x)$,
which we know by Lemma~\ref{lem:5} cannot vanish, a contradiction. In a similar fashion one excludes the case $\lambda\in (\lambda_0+1-\epsilon,\lambda_0+1)$. Therefore, $\lambda \in (\lambda_0+\epsilon,\lambda_0+1-\epsilon)$ and the claim follows by ($\chi.1$). 
\end{proof}

Using the functional $\A_H$ we can give the following equivalent formulation of the Arnold conjecture on $\C\PP^n$: the functional $\A_H$ has at least $n+1$ distinct $S^1$-families of 
critical points. In the next section we will show how to prove such a result using the relative cup-length of the Conley index.

%%%%%%%%%%%%%%
%%%%%%%%%%%%%%
%%%%%%%%%%%%%%

\section{IA-homotopies and the Conley index}
Throughout this section we will consider families of vector fields on $E:=H^{1/2}(\T,\R^{2n+2})$ which are 1-homogeneous and of the form 
\begin{equation}
T = L + K,
\label{eq:admissiblevf}
\end{equation}
where $L$ is as in \eqref{L} and $K:E\to E$ is compact (i.e. maps bounded sets into pre-compact sets) 
and admits a Lipschitz continuous extension to $L^2(\T,\R^{2n+2})$. 
Typically, $K$ is of the form $\jmath^* b$ for some Lipschitz continuous $b:L^2(\T,\R^{2n+2})\to L^2(\T,\R^{2n+2})$.

We will show that - under certain invertibility 
assumptions on the end-points of the family - the Conley index of the maximal 
bounded invariant set for the associated flow depends only on suitable homotopy classes of the end-points.
For notational convenience we will hereafter call vector fields as in \eqref{eq:admissiblevf} \textit{admissible}. We further say that an 
admissible vector field $T$ is \textit{non-vanishing} if $0\in E$ is the only zero of $T$, namely if $T(x)=0$ implies that $x=0$. 

\begin{rem}
For fixed $\lambda\in [\lambda_0,\lambda_0+1]$, the first component of $\nabla \A_H(\cdot, \lambda)$ is admissible. Moreover, the first component of $\nabla \A_H(\cdot,\lambda_0)$ resp.
$\nabla \A_H(\cdot,\lambda_0+1)$ is by assumption non-vanishing. 
\qed
\end{rem}

Let $\mathcal T:[0,1]\times E\times [\lambda_0,\lambda_0+1]\to E$ be a two-parameter-family of admissible vector fields such that $\mathcal T(s,\cdot,\lambda_0)$ and $\mathcal T(s,\cdot,\lambda_0+1)$ are 
invertible for all $s\in [0,1]$. The proof of Lemma~\ref{lem:5} goes through word by word showing that there exists $\epsilon>0$ such that 
\begin{equation}
\label{eq:second}
\inf_{\|x\|=1} \|\mathcal T(s,x,\lambda) \| >\epsilon, \qquad \forall s\in [0,1], \ \forall \lambda \in [\lambda_0,\lambda_0+\epsilon)\cup (\lambda_0+1-\epsilon, \lambda_0+1].
\end{equation}
For such an $\epsilon>0$ we choose a smooth function $\chi:\R\to \R$ satisfying ($\chi.0$)-($\chi.3$) and define
\begin{equation}
g(x,\lambda ) := \pi (\chi'(\lambda)\|x\|_2^2 - 1).
\label{eq:glambda}
\end{equation}

\begin{lm}
\label{lem:7}
Let $\{(s_m,x_m,\lambda_m)\}\subset [0,1]\times E\times [\lambda_0,\lambda_0+1]$ be a sequence such that 
$$\{(\mathcal T(s_m,x_m,\lambda_m), g(x_m,\lambda_m))\}\subset E\times \R$$
is bounded. Then $\{x_m\}$ is bounded in $E$. 
\end{lm}
\begin{proof}
If $\{\lambda_{m_k}\}$ has a subsequence which is entirely contained in $[\lambda_0,\lambda_0+\epsilon]\cup [\lambda_0+1-\epsilon,\lambda_0+1]$, then the boundedness of $\mathcal T(s_{m_k},x_{m_k},\lambda_{m_k})$ together with the homogeneity assumption and \eqref{eq:second} implies
$$C > \|\mathcal T(s_{m_k},x_{m_k},\lambda_{m_k})\| = \|x_{m_k} \| \cdot \left \| \mathcal T\left (s_m, \frac{x_{m_k}}{ \|x_{m_k}\|},\lambda_{m_k}\right )\right \| > \epsilon \| x_{m_k}\|$$
and hence that $\|x_{m_k}\|$ is bounded. If instead $\{\lambda_{m_k}\}\subset (\lambda_0+\epsilon,\lambda_0+1-\epsilon)$, then $\chi'(\lambda_{m_k})=1$ for all $k\in \N$, and hence the boundedness
of $g(x_{m_k},\lambda_{m_k})$ is equivalent to the boundedness of $\{x_{m_k}\}$ in $L^2$. Therefore, the particular form of $\mathcal T(s,x,\lambda)$ implies that 
$$C>\|\mathcal T(s_{m_k},x_{m_k},\lambda_{m_k})\| \geq \|Lx_{m_k}\| - \| K(s_{m_k},x_{m_k},\lambda_{m_k})\|\geq \|Lx_{m_k}\| - c,$$
that is $\{Lx_{m_k}\}\subset E$ is bounded. The claim follows, as $Lx_{m_k}=x_{m_k}^+-x_{m_k}^-$, and $\{x_{m_k}\}$ is bounded in $L^2$ (hence, in particular $\{x_{m_k}^0\}$ is bounded in $\R^{2n+2}$).
\end{proof}

\begin{cor}
\label{cor:1}
\textit{Let $\mathcal T:[0,1]\times E\times [\lambda_0,\lambda_0+1]\to E$ be as above. Then there exists a maximal bounded invariant set, that is, there exists $R_0>0$ such that the following holds: $B(R_0)\times [\lambda_0,\lambda_0+1]$ is an isolating neighborhood 
for the flow $\eta_s$ generated by $(\mathcal T_s,g)$, $\mathcal T_s:=\mathcal T(s,\cdot,\cdot)$, for every $s\in [0,1]$, and}
$$\text{inv}\, (\eta_s, B(R)\times [\lambda_0,\lambda_0+1]) = \text{inv}\, (\eta_s, B(R_0)\times [\lambda_0,\lambda_0+1]), \quad \forall R\geq R_0.$$
\end{cor}
\begin{proof}
See \cite[Prop. 2.4]{St} or \cite[Prop. 2.14]{StWa}.
\end{proof}

\begin{rem}
For every $s\in [0,1]$, the $\eta_s$-flow line starting at $(0,\lambda)$, for some $\lambda\in \R$, is unbounded as 
$$\big ( \mathcal T_s (0,\lambda), g(0,\lambda) \big ) = (0,-\pi).$$ 
This means that $(0,\lambda)\not \in \text{inv}\, (\eta_s, B(R)\times [\lambda_0,\lambda_0+1])$ for every $\lambda\in \R$, for every $s\in [0,1]$. Since the maximal bounded invariant set is compact (see \cite[Prop. 2.3]{gip}), 
this implies that $\text{inv}\, (\eta_s, B(R)\times [\lambda_0,\lambda_0+1])$ has distance to $\{0\}\times \R\subset E\times \R$ bounded away from zero by some positive constant, say $r_0>0$. 
Therefore, we can replace the isolating neighborhood $B(R_0)\times [\lambda_0,\lambda_0+1]$ by 
$$A(r_0,R_0)\times [\lambda_0,\lambda_0+1],$$
where 
$$A(r_0,R_0):= \big \{ x\in E \ \big |\ r_0\leq \|x\| \leq R_0\big \}$$
is the annulus in $E$ with inner radius $r_0$ and outer radius $R_0$. \qed
\end{rem}

\begin{df}
Two non-vanishing admissible vector fields $T_0,T_1$ are \textit{IA-homotopic} if there exists a continuous map $\mathcal T:[0,1]\times E \to E$ such that $\mathcal T(0,\cdot) =T_0$, $\mathcal T(1,\cdot)=T_1$, and $\mathcal T(s,\cdot)$ is non-vanishing and admissible for every $s\in [0,1]$. 
\end{df}

\begin{prop}
\label{prop:1}
The Conley index of the maximal bounded invariant set for a flow $\eta$ on $E\times \R$ generated by a vector field of the form $(\mathcal T(x,\lambda), g(x,\lambda))$, where $\mathcal T:E\times \R\to E$ is a family of admissible 
vector fields with $T_{\lambda_0}:=\mathcal T(\cdot,\lambda_0)$ and $T_{\lambda_0+1}:=\mathcal T(\cdot,\lambda_0+1)$ non-vanishing,
 depends only on the IA-homotopy classes of $T_{\lambda_0}$ and $T_{\lambda_0+1}$. 
\end{prop}

\begin{proof}
Let $\mathcal T,\mathcal T'$ be two families as in the statement of the proposition, and suppose that $\mathcal H_{\lambda_0}$ resp. $\mathcal H_{\lambda_0+1}$ are 
IA-homotopies between $T_{\lambda_0}$ and $T_{\lambda_0}'$ resp. $T_{\lambda_0+1}$ and $T_{\lambda_0+1}'$.
In the first step we homotope $\mathcal T$ linearly to 
$$\widetilde{\mathcal T}(x,\lambda ):= (\lambda_0+1-\lambda)T_{\lambda_0}(x) + (\lambda-\lambda_0)T_{\lambda_0+1}(x)$$
and $\mathcal T'$ linearly to 
$$\widetilde{\mathcal T}'(x,\lambda ):= (\lambda_0+1-\lambda)T_{\lambda_0}'(x) + (\lambda-\lambda_0)T_{\lambda_0+1}'(x).$$
Notice that by construction such homotopies are through admissible vector fields which are non-vanishing at the end-points. In the second step, we use the IA-homotopies
$\mathcal H_{\lambda_0}$ and $\mathcal H_{\lambda_0+1}$ to obtain a homotopy between $\widetilde{\mathcal T}$ and $\widetilde{\mathcal T}'$. By construction, Lemma~\ref{lem:7} and Corollary~\ref{cor:1} yield now a continuation between $\mathcal T$ and $\mathcal T'$.
\end{proof}

%%%%%%%%%%%

\section{Proof of the main theorem}

\subsection{The $C^0$-small case} In this subsection we sketch the proof Theorem~\ref{thm:main} in the particular case of Hamiltonians which are sufficiently small in the $C^0$-norm. We will not provide full details, since we will prove Theorem~\ref{thm:main} rigorously in the next subsection for an arbitrary Hamiltonian. However, we think it is important to 
highlight the main steps of the proof in such an easier case since, as we will see below, the general case will reduce after
applying several successive IA-homotopies to a situation which can be dealt with with analogous ideas.

Suppose first that $H_0$ vanish identically. Then, we readily see that 
$$F_\lambda(x) = - J \dot x + 2\pi \lambda x$$ 
does not vanish for $\lambda \in \R\setminus \Z$, unless $x=0$. In particular, $\jmath^* F_\lambda$ is admissible and non-vanishing for every $\lambda\in \R\setminus \Z$. 
Choose for instance $\lambda_0=-1/2$. Comparing $\jmath^*F_{-1/2}$ with $\jmath^* F_{1/2}$, we see
that they have positive scalar product on $E^+\oplus E^-$, and that 
$$\jmath^* F_{-1/2} |_{E^0} = - \pi \cdot \text{id}, \quad \jmath^*F_{1/2} |_{E^0} = \pi \cdot \text{id}.$$
For this reason, the linear homotopies from 
$$\nabla \mathbb A_0 (x,\lambda) = \big (\jmath^*(-J\dot x + 2\pi \chi(\lambda) x), \pi (\chi'(\lambda) \|x\|_2^2 -1 ) \big )$$
to 
$$\big (\jmath^*(-J\dot x + 2\pi \chi(\lambda) x^0), \pi (\chi'(\lambda) \|x\|_2^2 -1 ) \big ) = \big ( x^+ -x^- + 2\pi \chi(\lambda)x^0, \pi (\chi'(\lambda) \|x\|_2^2 -1 ) \big )$$
and then to 
$$\big (x^+ -x^- + 2\pi \chi(\lambda)x^0, \pi (\chi'(\lambda) |x^0|^2 -1 ) \big )$$
define a continuation (c.f. Lemma~\ref{lem:7} and Corollary~\ref{cor:1} above), where $x=x^++x^-+x^0$ denotes the canonical orthogonal splitting of $x\in E\cong E^+\oplus E^-\oplus E^0$. The last vector field can be seen as a product vector field in 
$$(E^\oplus \oplus E^- ) \times  (E^0\oplus\R)$$
and it is not hard to see that the set 
$$B(E^+\oplus E^-) \times \big ( A(1/2,2) \times [-1/2,1/2]\big )$$
is a product isolating neighborhood for the induced flow. Here, $B(E^+\oplus E^-)$ denotes the unit ball in $E^+\oplus E^-$, and $A(1/2,2)$ the annulus in $E^0$ with inner radius $1/2$ and outer radius $2$. 
The Conley index for the flow on $E^+\oplus E^-$ is non-trivial since it is a sphere spectrum.
An index pair for the flow on $E^0\oplus \R$ can be explicitly written as 
$$(N,L) = \Big ( A(1/2,2) \times [-1/2,1/2], A(1/2,2)\times \{-1/2\} \cup S(1/2)\times [-1/2,0] \cup S(2)\times [0,1/2]\Big ),$$
where here $S(r)$ denotes the sphere of radius $r$ in $E^0$.
Since the index pair is $S^1$-invariant, we can quotient it out, and the resulting sets $N/S^1$ and $(N/S^1)/(L/S^1)$ have the homotopy types of $\C \mathbb P^n$ and $\C\mathbb P^n\wedge S^1$, the smash product 
of $\C\mathbb P^n$ and $S^1$, respectively, and the relative cup-length is equal to $n+1$. The product formula for the relative cup-length given in Lemma~\ref{lem:3}, 3) together with the invariance under continuation 
yields now that the relative cup-length of the flow generated by $\nabla \A_0$ is at least $n+1$. 

Let now $H_0:\T\times \C\mathbb P^n\to \R$ be a smooth Hamiltonian such that $\|H_0\|_{\infty}<\pi/2$. 
A computation shows that also in this case $\jmath^*F_{-1/2}$ and $\jmath^* F_{1/2}$ are admissible and non-vanishing (roughly speaking, one sees that the only part of $\nabla H$, $H$ being the lift of $H_0$ to $\R^{2n+2}$, that can ``cancel out'' with the term $\pm \pi x$ thus making the vector fields $\jmath^* F_{\pm 1/2}$ vanishing is $2H(x)$). Also, since for every $s\in [0,1]$ the Hamiltonian $s\cdot H_0$ satisfies the 
condition $\|s\cdot H_0\|_{\infty}< \pi/2$ as well, we 
see that the linear homotopy from $H$, the lift of $H_0$ to $\R^{2n+2}$, to the trivial Hamiltonian yields a continuation (indeed, $\jmath^* F_{-1/2}$ and $\jmath^* F_{1/2}$ are admissible and non-vanishing 
throughout the whole homotopy). The claim follows now from the case of the trivial Hamiltonian treated before. 

Before turning our attention to the case of general Hamiltonians we shall notice that, for a Hamiltonian which does not satisfy the condition $\|H_0\|_{\infty}<\pi/2$, the linear homotopy $s\mapsto s\cdot H_0$ does 
not yield a continuation in general. This can be seen already taking 
$$H_0(t,x) \equiv \frac pi 2 \qquad (\text{i.e.} \ H(t,x) = \frac \pi 2\cdot |x|^2) .$$
Indeed, there is no $\lambda_0\in \R$ such that $\jmath^* F_{\lambda_0}$ is non-vanishing throughout the whole linear homotopy. In particular, the argument above cannot be extended 
to arbitrary Hamiltonians. Nevertheless, we will see in the next subsection that even in the general case, after suitably changing $\jmath^* F_{\lambda_0}$ and $\jmath^* F_{\lambda_0+1}$ within their IA-homotopy classes, 
we can reduce to a situation similar to the one of $C^0$-small Hamiltonians. More precisely, the modified vector fields will have a product form and 
coincide on ``high modes'' and will be equal to $L$, there will be precisely one mode (a $2n+2$-dimensional 
subspace of $E$ of the form $E_N$ for some $N\in \Z$) for which the first one will be given by -id and the second one by id, and they will also coincide for all other ``low modes''. Even though we cannot explicitly compute 
the Conley index for the flow in the low modes, we will be able to show that the Conley index is non-trivial, and this will be enough to deduce the Arnold conjecture in a similar way as above, namely using the product 
formula for the relative cup-length and the invariance under continuation.

%%%%%%%%%%%%

\subsection{The general case}

%Goal of this subsection is to better understand the Conley index of the maximal bounded invariant set for the flow $\eta$ generated by $\nabla \A_H$ in the case of a Hamiltonian $H_0$ which is not necessarily small in the $L^\infty$-norm. To do this, 
%we will employ the abstract result proved in Subsection 4.1 to change $\jmath^* F_{\lambda_0}$ and $\jmath^* F_{\lambda_0+1}$ within their 
%IA-homotopy classes to vector fields which are more suitable for computations. 
%Even though we will not be able to determine the Conley index explicitly, this will give us enough information to be able to conclude the (degenerate) 
%Arnold conjecture on $\C\mathbb P^n$. 

Recall that 
$$\nabla \A_H (x,\lambda) = \Big ( \jmath^* \big (-J\dot x - \nabla H(x) + 2\pi \chi(\lambda) x\big ), g(x,\lambda)\Big ),$$
where $\chi:\R\to \R$ satisfies ($\chi.0$)-($\chi.3$) and the function $g$ is given by~\eqref{eq:glambda}. In particular 
$$\nabla \A_H (\cdot ,\lambda_0) = (\jmath^* F_{\lambda_0}(\cdot), -\pi ), \qquad \nabla \A_H (\cdot ,\lambda_0+1 ) = (\jmath^* F_{\lambda_0+1}(\cdot), -\pi ).$$
By construction the vector fields $\jmath^* F_{\lambda_0}=L+\jmath^*b_{\lambda_0}$ and $\jmath^* F_{\lambda_0+1}=L+\jmath^*b_{\lambda_0+1}$,
$$b_\lambda(x) := -\nabla H(x) + 2\pi \chi(\lambda)x,$$
 are admissible and non-vanishing; in particular, $0\in E$ is the 
only critical point of the restriction of $\A_H$ to $\{\lambda=\lambda_0\}$ and $\{\lambda=\lambda_0+1\}$ which will be denoted respectively by 
$\A_H(\cdot,\lambda_0)$ and $\A_H(\cdot,\lambda_0+1)$. We denote with $h_{\lambda_0}$ the 
Conley index corresponding to $\A_H(\cdot,\lambda_0)$.  

\begin{lm}
The Conley index $h_{\lambda_0}$ is non-trivial. 
\label{lem:8}
\end{lm}
\begin{proof}
Let $n_0\in \N$ be as in Lemma~\ref{lem:2}. Clearly, it is enough to show that $h_{\lambda_0}^{-n_0,n_0}$ is 
non-trivial. Since $\A_H(\cdot,\lambda_0)$ is $S^1$-invariant, its gradient is $S^1$-equivariant, hence in particular $\Z_2$-equivariant (i.e. odd). Denoting by $\chi(h_{\lambda_0}^{-n_0,n_0})$ 
the Euler-characteristic of $h_{\lambda_0}^{-n_0,n_0}$, a result of McCord (see \cite{McCord}) implies that 
$$\chi(h_{\lambda_0}^{-n_0,n_0}) = \text{deg} (\nabla_x \A_H (\cdot, \lambda_0)\Big |_{E^{-n_0,n_0}}) \neq 0$$
as the degree of an odd vector field is odd. 
\end{proof}

We are now ready to homotope $\jmath^* F_{\lambda_0}$ and $\jmath^* F_{\lambda_0+1}$ within their 
IA-homotopy classes to vector fields which are more accessible for computations. This will be done in several steps. 
Roughly speaking we will homotope $\jmath^* F_{\lambda_0}$ and $\jmath^* F_{\lambda_0+1}$ to vector fields $V_{\lambda_0}$ and $V_{\lambda_0+1}$ respectively which are of ``product type'' and 
identical on all ``modes'' $E_k$ besides exactly one, namely $E_{-N-1}$ for $N\in \N$ large enough, where $V_{\lambda_0}$ is given by -id and $V_{\lambda_0+1}$ is given by id. Moreover, both vector fields 
coincide with $L$ on $E_{-N-1,N+1}^\perp$. This implies that the relative cup-length of the Conley index of $\nabla \A_H$ can be computed by passing to the finite dimensional approximation in $E_{-N-1,N+1}$ and using the product formula.

We notice already at this point, and we will not mention it again, that the isolating neighborhood will be of the form 
$A(r_0,R_0)\times [\lambda_0,\lambda_0+1]$, for suitable $r_0<R_0$, throughout all steps. 

\vspace{2mm}

\textbf{Step 1:} We change $\jmath^* F_{\lambda_0+1}$ within its IA-homotopy class to the conjugation of $\jmath^*F_{\lambda_0}$ by a suitable shift operator. To do this, 
we start comparing $F_{\lambda_0}$ and $F_{\lambda_0+1}$. We define the shift operator
$$\text{Sh}:E\to E,\qquad x \mapsto e^{2\pi J t} x$$
and readily compute using the $S^1$-equivariance of $\nabla H$
 \begin{align*}
F_{\lambda_0} ( \text{Sh} (x)) &= - J \dot{(\text{Sh} (x))} - \nabla H(\text{Sh} (x)) + 2 \pi \lambda_0 \text{Sh}(x) \\
			& = \text{Sh} (-J \dot x ) + 2\pi \text{Sh(x)} - \text{Sh}(\nabla H(x)) + 2\pi \lambda_0 \text{Sh}(x)\\
			& = \text{Sh} ( - J \dot x - \nabla H(x) + 2\pi (\lambda_0+1) x )\\
			&= \text{Sh}(F_{\lambda_0+1}(x)),
\end{align*}
which can be equivalently written as 
$$F_{\lambda_0+1} = \text{Sh}^{-1} \circ F_{\lambda_0} \circ \text{Sh}.$$
Hence,
\begin{equation}
\jmath^* F_{\lambda_0+1} = \jmath^*( \text{Sh}^{-1} \circ F_{\lambda_0} \circ \text{Sh}).
\label{eq:conjugationflambda}
\end{equation}
In other words, $\jmath^* F_{\lambda_0+1}$ is almost the conjugation of $\jmath^* F_{\lambda_0}$ by the shift operator Sh: this would indeed be the case if $\jmath^*$ and Sh commuted.
Even though this is not the case, we show now that we can homotope $\jmath^* F_{\lambda_0+1}$ within its IA-homotopy class to the conjugation of $\jmath^*F_{\lambda_0}$ by Sh:
$$\text{Sh}^{-1}\circ (\jmath^* F_{\lambda_0}) \circ \text{Sh}.$$
To do that we consider the linear combination 
\begin{equation}
s \mapsto \mathcal T_s := s \cdot \jmath^*( \text{Sh}^{-1} \circ F_{\lambda_0} \circ \text{Sh}) + (1-s) \cdot \text{Sh}^{-1}\circ (\jmath^* F_{\lambda_0}) \circ \text{Sh}.
\label{eq:IAhomotopy}
\end{equation}
Notice that $\mathcal T_s$ is admissible for every $s\in [0,1]$, as
$$\text{Sh}^{-1}\circ (\jmath^* F_{\lambda_0}) \circ \text{Sh} = \text{Sh}^{-1} \circ L \circ \text{Sh} +  \text{Sh}^{-1} \circ \jmath^* b_{\lambda_0} \circ \text{Sh},$$
and it is straightforward to check that $\text{Sh}^{-1}\circ L \circ \text{Sh}$ and $L$ coincide up to a finite rank operator (indeed, they coincide on $(E^{-1,0})^\perp$). %For future reference we write 
%$$\text{Sh}^{-1}\circ (\jmath^* F_{\lambda_0}) \circ \text{Sh} = L + K'_{\lambda_0+1}.$$
Thus, in order to show that \eqref{eq:IAhomotopy} defines an IA-homotopy it suffices to prove that $\mathcal T_s$ is non-vanishing
for every $s\in [0,1]$.  To do this, we observe that, since 
$$\mathcal T_s = [s\cdot \jmath^* \circ \text{Sh}^{-1} + (1-s) \cdot  \text{Sh}^{-1}\circ \jmath^* ] \circ F_{\lambda_0}\circ \text{Sh}$$
and $F_{\lambda_0}$ and $\text{Sh}$ are non-vanishing, all we need to show is that the linear operator
$$s\cdot \jmath^* \circ \text{Sh}^{-1} + (1-s) \cdot  \text{Sh}^{-1}\circ \jmath^*$$
is non-vanishing. To do that, we compute its values on a Hilbert basis: denoting by $e_k$ any basis element of $E_k$ (which is a $2n+2$-dimensional vector space) and recalling that 
$$\jmath^* e_k = \frac{1}{|k|} e_k, \quad \forall k\neq 0, \qquad \text{and}\ \ \jmath^* e_0 = e_0,$$
we compute for $k\neq 0,1$
\begin{equation}
(s\cdot \jmath^* \circ \text{Sh}^{-1} + (1-s) \cdot  \text{Sh}^{-1}\circ \jmath^*)e_k =\Big ( \frac{s}{|k-1|} + \frac{1-s}{|k|} \Big ) e_{k-1} \neq 0 , \quad \forall s \in [0,1].
\label{eq:nonvanishing1}
\end{equation}
For $k=0,1$ we have instead
\begin{equation}
(s\cdot \jmath^* \circ \text{Sh}^{-1} + (1-s) \cdot  \text{Sh}^{-1}\circ \jmath^*)e_k=e_{k-1}.
\label{eq:nonvanishing2}
\end{equation}
From \eqref{eq:nonvanishing1} and \eqref{eq:nonvanishing2} we readily see that $\mathcal T_s$ is non-vanishing for every $s\in [0,1]$. 
Proposition~\ref{prop:1} therefore implies that the Conley index of the maximal bounded invariant set for $\nabla \A_H$ is equal to the Conley index of the maximal bounded invariant set for 
$(\mathcal T(x,\lambda), g(x,\lambda))$, where $\mathcal T$ is given by 
$$\mathcal T (x,\lambda) := (\lambda_0+1-\lambda) \cdot \jmath^* F_{\lambda_0} + (\lambda-\lambda_0)\cdot \text{Sh}^{-1}\circ \jmath^* F_{\lambda_0}\circ \text{Sh}.$$

\textbf{Step 2:} We change the endpoints of $\mathcal T$, namely $\jmath^*F_{\lambda_0}$ and $\text{Sh}^{-1}\circ \jmath^* F_{\lambda_0}\circ \text{Sh}$, within their IA-homotopy class in such a 
way that the new endpoints coincide with $L$ for high modes (roughly speaking, on $(E^{-N,N})^\perp$ for some $N\in \N$ large enough). This is possible since both endpoints are of the form $L+K$ for some 
compact $K:E\to E$. Therefore, we can find $n_0\in \N$ such that both $\jmath^*F_{\lambda_0}$ and $\text{Sh}^{-1}\circ \jmath^* F_{\lambda_0}\circ \text{Sh}$ have positive scalar product\footnote{Such an $n_0$ is the 
same one appearing in Lemmas~\ref{lem:1}, \ref{lem:2}, and \ref{lem:4}.} with $L$ on $(E^{-n_0,n_0})^\perp$. We thus choose $N>n_0+1$ and homotope $\jmath^*F_{\lambda_0}$ within its IA-homotopy class linearly to 
$$V_{\lambda_0}:= L + \pi^{-N,N}\circ \jmath^* b_{\lambda_0} \circ \pi^{-N,N}$$
and $\text{Sh}^{-1}\circ \jmath^* F_{\lambda_0}\circ \text{Sh}$ to 
$$\widetilde V_{\lambda_0+1} := \text{Sh}^{-1}\circ V_{\lambda_0} \circ \text{Sh}.$$
Notice that $V_{\lambda_0}\equiv L$ on $(E^{-N,N})^\perp$, whereas $\widetilde V_{\lambda_0+1}\equiv L \equiv \text{Sh}^{-1}\circ L \circ \text{Sh}$ on $(E^{-N-1,N-1})^\perp$. 

\vspace{2mm}

\textbf{Step 3:} We unshift the vector field $\widetilde V_{\lambda_0+1}$. To do that we set $M\in GL(E)$ by 
$$Mx := \left \{\begin{array}{r} \text{Sh} (x) \qquad \qquad \qquad x \in E^{-N-1,N}, \\ \text{Sh}^{-2(N+1)}(x)  \  \qquad \ \qquad  \ x \in E_{N+1},\\ x \qquad \qquad x \in (E^{-N-1,N+1})^\perp,\end{array}\right .$$
and consider a path $\sigma :[0,1]\to GL(E)$ connecting id to $M$ such that\footnote{This is possible since $M|_{E^{-N-1,N+1}}\in SO(E^{-N-1,N+1})$.} $\sigma \equiv$ id on $(E^{-N-1,N+1})^\perp$.
Conjugating $V_{\lambda_0+1}$ by $\sigma$ defines an IA-homotopy between $\widetilde V_{\lambda_0+1}$ and $V_{\lambda_0+1}:=\sigma(1)\circ \widetilde V_{\lambda_0+1}\circ \sigma(1)^{-1}$, and a straightforward 
computation shows that 
$$ V_{\lambda_0+1} x = \left \{\begin{array}{r} V_{\lambda_0} x \qquad \qquad \qquad x \in E^{-N,N}, \\ x \qquad \qquad \qquad \ \  x \in E_{-N-1}, \\ Lx \qquad \qquad x \in (E^{-N-1,N})^\perp.\end{array}\right.$$
The invariance under IA-homotopies (c.f. Proposition~\ref{prop:1}) implies once again that we can compute the Conley index using the vector field $(\mathcal T(x,\lambda),g(x,\lambda))$ where now $\mathcal T$ is given by
$$\mathcal T(x,\lambda ) := (\lambda_0+1-\lambda) \cdot V_{\lambda_0} + (\lambda-\lambda_0)\cdot V_{\lambda_0+1}.$$
An easy inspection of the explicit formula for $V_{\lambda_0+1}$ shows that $V_{\lambda_0+1}$ coincides with $V_{\lambda_0}$ everywhere besides on $E_{-N-1}$, 
on which we have $V_{\lambda_0}=-$id and $V_{\lambda_0+1}=\, $id. This implies that the vector field $\mathcal T$ can be more conveniently written using the orthogonal splitting 
$$E \cong E^{-N,N} \oplus E_{-N-1} \oplus \bigoplus_{k \in \Z \setminus \{-N-1,...,N\}} E_k \ \  \ni (y,z,w)$$
as 
$$\mathcal T( (y,z,w),\lambda ) = \big ( V_{\lambda_0} y, (-1+2(\lambda-\lambda_0))z,Lw\big ).$$

\textbf{Step 4:} We consider the family of vector fields
$$(\mathcal T((y,z,w),\lambda), g((s\cdot y, z,s\cdot w),\lambda))$$
to homotope the function $g((y,z,w),\lambda)$ linearly to $(g(0,z,0),\lambda)$.
To do this, all we need to show  is that there exists an isolating neighborhood which is common for every $s\in [0,1]$. In virtue of \cite[Prop 2.4]{St} (see also \cite[Prop. 2.14]{StWa}), as in Lemma~\ref{lem:7}
it suffices to prove the following: if 
$\{((y_m,z_m,w_m), \lambda_m,s_m)\}\subset E\times [\lambda_0,\lambda_0+1]\times [0,1]$ is a sequence such that 
$$\{(\mathcal T((y_m,z_m,w_m),\lambda_m), g((s_m\cdot y_m, z_m,s_m\cdot w_m),\lambda_m)) \}\subset E\times \R$$
is bounded, then $\{(y_m,z_m,w_m)\}$ is bounded in $E$. We start noticing that 
$$ \| \mathcal T((y_m,z_m,w_m),\lambda_m)\|^2 = \|V_{\lambda_0}y_m\|^2 + |-1+2(\lambda_m-\lambda_0)| \cdot \|z_m\|^2 + \|w_m\|^2$$
and this immediately implies that $\{y_m\}$ and $\{w_m\}$ are bounded (recall that $V_{\lambda_0}$ is homogeneous and bounded away from zero on the sphere of radius one). To show that $\{z_m\}$ is bounded, we observe that the 
boundedness of  
$$g((s_m\cdot y_m,z_m,s_m\cdot w_m),\lambda_m)=\pi \big (\chi'(\lambda_m) \big [ |s_m|^2 \cdot (\|y_m\|_2^2 + \|w_m\|_2^2) + \|z_m\|_2^2\big ] -1 \big )$$
implies that $\{z_m\}$ is bounded in $L^2$, which is equivalent to boundedness in $H^{1/2}$ as $\{z_m\}\subset E_{-N-1}$. Therefore, to compute the Conley index we can use the vector field 
$$\big (\mathcal T((y,z,w),\lambda), g(z,\lambda)\big ),$$
where for notational convenience we set $g(z,\lambda):=g((0,z,0),\lambda).$

\vspace{2mm}

\textbf{Step 5:} We change $\big (\mathcal T((y,z,w),\lambda), g(z,\lambda)\big )$ to its finite dimensional approximation, namely considering the flow on $(E^{-N,N}\oplus E_{-N-1})\times \R$ 
generated by the vector field
$$\big ((V_{\lambda_0} y, (-1+2(\lambda-\lambda_0))z), g(z,\lambda)\big ).$$
This is possible since flow lines with $w\neq 0$ are unbounded in positive and negative direction, being $\mathcal T$ equal $L$ on the $w$-direction. In particular, all critical points of $\mathcal T$ 
are in fact contained in $E^{-N,N}\oplus E_{-N-1}$. The Conley index is here computed by taking 
\begin{align*}
\Omega& := A^{-N-1,N}(r_0,R_0)\times [\lambda_0,\lambda_0+1]\\ & := \big \{(y,z)\in E^{-N,N}\times E_{-N-1}\ \big |\ r_0\leq \|y\|^2 + \|z\|^2 \leq R_0\big \} \times [\lambda_0,\lambda_0+1]
\end{align*}
as isolating neighborhood. We shall also notice that replacing $N$ with some larger $N'$ only changes the Conley index by the suspension by a sphere of suitable dimension, 
and hence in particular the relative cup-length does not change (see \cite[Thm. 2.5]{KGUss}). 

\vspace{2mm}

\textbf{Step 6:} We finally change the isolating neighborhood $\Omega$ to a product isolating neighborhood, namely to 
$$\Omega':= B(r) \times A_{-N-1}(r',R')\times [\lambda_0,\lambda_0+1]\subset \Omega$$
with $A_{-N-1}(r',R') := \{z\in E_{-N-1}\ |\ r'\leq \|z\|\leq R'\}$ for suitable $r,r',R'>0$, in order to employ the formula for the relative cup-length of a product flow 
with product isolating neighborhood, see Lemma~\ref{lem:3}, Item (3). As we show now, replacing $\Omega$ with $\Omega'$ does not change the relative cup-length when passing to the quotient: 
to see this we first observe that
$\Omega$ and $\Omega'$ are $S^1$-equivariantly homotopy equivalent to $S^{(2N+2)(2n+2)-1}$ and $S^{2n+1}$ respectively. Now, 
the quotients $\Omega/S^1$ and $\Omega'/S^1$ are homotopy equivalent to 
 $\mathbb C\mathbb P^{(2N+2)n}$ and $\mathbb C\mathbb P^n$ respectively, and the claim readily follows by Lemma~\ref{lem:3}, Item (2), since the inclusion induces a surjective homomorphism 
 from the cohomology of $\mathbb C\mathbb P^{(2N+2)n}$ to the cohomology of $\mathbb C\mathbb P^n$. 
 
\begin{proof}[Proof of Theorem~\ref{thm:main}]
By contradiction assume that $\A_H$ have $k\leq n$ distinct $S^1$-families of critical points. Then we can find $k$ disjoint isolating neighborhoods whose projection under the $S^1$-action is contractible. 
Such isolating neighborhoods yield in virtue of Lemma~\ref{lem:4} pairwise disjoint contractible isolating neighborhoods for the finite dimensional approximation to $E^{-N-1,N+1}\times [\lambda_0,\lambda_0+1]$.
This by Theorem~\ref{thm:1} would imply for the relative cup-length that 
$$\mathcal Y( \Omega/S^1, \bar \eta^{-N-1,N+1}) \leq k \leq n,$$
where $\bar \eta^{-N-1,N+1}$ denotes the projection of the flow $\eta^{-N-1,N+1}$ on the orbit space. 
On the other hand, since the Conley index $h_{\lambda_0}^{-N-1,N}$ is non-trivial (see Lemma~\ref{lem:8}), by Lemma~\ref{lem:3}, Item (3), we have that $\mathcal Y( \Omega/S^1, \bar \eta^{-N-1,N+1})$ is bounded from 
below by the cup-length of $\mathbb C\mathbb P^n$ plus 1, namely $n+1$ (the construction of an explicit index pair can be found in \cite[Lemma 6.2]{KGUss}; see also Subsection~4.1). This clearly yields a contradiction and completes the proof. 
\end{proof}

%We are now ready to consider finite dimensional approximations. 
%Thus, let $n_0\in \N$ be such that the Conley indices of the finite dimensional approximations of $\jmath^*F_{\lambda_0}$ and $\jmath^*F_{\lambda_0+1}$ 
%stabilize as in the statement of 
%Lemma~\ref{lem:2}. We choose $k,l>n_0+1$ and denote by $h^{-k,l}_{\lambda_0}$ resp. $h^{-k,l}_{\lambda_0+1}$ the Conley index of the projection of 
%$\jmath^* F_{\lambda_0}$ resp. $\jmath^* F_{\lambda_0+1}$ to $E^{-k,l}$.
%
%\begin{lm}
%The Conley index $h_{\lambda_0+1}$ is the suspension of the Conley index $h_{\lambda_0}$ by $S^{2n+1}$, namely 
%$$h_{\lambda_0+1}=S^{2n+1}h_{\lambda_0}.$$
%\end{lm}
%
%\begin{proof}
%Noticing that $E^{-k,l}=\text{Sh} (E^{-k-1,l-1})$, we use the invariance of the Conley index under IA-homotopies to infer that 
%$$h_{\lambda_0}^{-k,l} = h_{\lambda_0+1}^{-k-1,l-1}.$$
%Lemma~\ref{lem:2} implies now that 
%$$h_{\lambda_0+1}^{-k,l}=S^{2n+1} h_{\lambda_0+1}^{-k-1,l-1}= S^{2n+1} h_{\lambda_0}^{-k,l}.$$
%\end{proof}

%.....

%%%%%%%%%%%
%%%%%%%%%%%
%%%%%%%%%%%

\end{document}